\documentclass[12pt]{amsart}

\usepackage[utf8]{inputenc}
\usepackage{ae,aecompl,aeguill}	%
\usepackage{mathtools}
\usepackage{amsmath}
\usepackage{amssymb}
\usepackage{mathrsfs}
\usepackage{amscd}
\usepackage{hyperref}
\usepackage{tikz}
\usetikzlibrary{arrows}
\usepackage{comment}
\usepackage{subcaption}
\usepackage{lineno}

\makeatletter
\newskip\@bigflushglue \@bigflushglue = -100pt plus 1fil

\def\bigcentering{\let\\\@centercr\rightskip\@bigflushglue
\leftskip\@bigflushglue
\parindent\z@\parfillskip\z@skip}

\makeatother

\newcommand{\Q}{\mathbb{Q}}

\newcommand{\suchthat}{\mid}

\renewcommand{\L}{\mathcal L}

\newcommand{\mix}{{\operatorname{mix}}}
\newcommand{\rel}{{\operatorname{rel}}}
\newcommand{\sort}{{\operatorname{sort}}}

\newtheorem{theorem}{Theorem}[section]
\newtheorem{proposition}[theorem]{Proposition}

\newtheorem{corollary}[theorem]{Corollary}
\newtheorem{conjecture}[theorem]{Conjecture}
{\theoremstyle{definition}
}
{\theoremstyle{remark}
}
{\theoremstyle{remark}
\newtheorem{example}[theorem]{Example}
}
\newtheorem{problem}{Problem}

\numberwithin{equation}{section}

\begin{document}

\title[Random-to-random shuffling on linear extensions]{Spectral gap for random-to-random shuffling on linear extensions}

\author[A. Ayyer]{Arvind Ayyer}
\address{Department of Mathematics, Indian Institute of Science, Bangalore - 560012, India.}
\email{arvind@math.iisc.ernet.in}

\author[A. Schilling]{Anne Schilling}
\address{Department of Mathematics, UC Davis, One Shields Ave., Davis, CA 95616-8633, U.S.A.}
\email{anne@math.ucdavis.edu}

\author[N. M. Thi\'ery]{Nicolas M.~Thi\'ery}
\address{Univ Paris-Sud, Laboratoire de Recherche en Informatique,
  Orsay, F-91405; CNRS, Orsay, F-91405, France}
\email{Nicolas.Thiery@u-psud.fr}

\begin{abstract}
In this paper, we propose a new Markov chain which generalizes random-to-random shuffling on permutations to
random-to-random shuffling on linear extensions of a finite poset of size $n$. We conjecture that the second largest 
eigenvalue of the transition matrix is bounded above by $(1+1/n)(1-2/n)$ with equality when the poset is disconnected.
This Markov chain provides a way to sample the linear extensions of the poset with a relaxation time bounded
above by $n^2/(n+2)$ and a mixing time of $O(n^2 \log n)$. 
We conjecture that the mixing time is in fact $O(n \log n)$ as for the usual random-to-random shuffling.
\end{abstract}

\maketitle

\section{Introduction}

The random-to-random shuffle removes a card at a uniformly random
position from a deck of $n$ distinct cards and replaces it in the deck at another uniformly random position. 
This gives rise to a random walk on the permutations of the symmetric group 
$S_n$ which has recently attracted a lot of attention. Random-to-random shuffling was originally described 
in~\cite[Section 8.2]{diaconis_saloff-coste.1995} and its mixing time 
is proved to be $O(n \log n)$ (see Section~\ref{section.mixing}). In his PhD
thesis, Uyemura-Reyes~\cite{reyes.2002} refined this result by providing an
upper bound of the order of $4n \log n$ and a lower bound of the order of $\frac{1}{2} n \log n$.
Uyemura-Reyes also gave some partial results and many conjectures about the eigenvalues of the transition matrix. 
Since then, the constants have been further improved in~\cite{saloff_coste_zuniga.2008,subag.2013, morris_qin.2014}.
One of the main motivations in the study of random-to-random shuffles is a conjecture of Diaconis~\cite{diaconis.2003} on 
the existence of a cutoff at $\frac{3}{4} n \log n$.
See~\cite{levin_peres_wilmer.2009} for details about the mixing time and the cutoff phenomenon.

In this paper we propose a generalization of the random-to-random shuffling Markov chain by extending it to linear extensions 
of a finite poset $P$ of size $n$. Denote the order of $P$ by $\preceq$ and 
assume that the vertices of $P$ are labeled by the integers in $[n]:=\{1,2,\ldots,n\}$.
A \emph{linear extension} of $P$ is a permutation $\pi$ in the symmetric group $S_n$
such that $\pi_{i} \prec \pi_{j}$ in $P$ implies $i<j$ as integers. We denote
by $\L(P)$ the set of all linear extensions of $P$.

Following~\cite{stanley.2009} we can now define analogues $\tau_i$ acting on
$\L(P)$ of the simple transpositions $s_i=(i,i+1)$ acting on the right
on $S_n$ (that is, on positions).  Namely, for $\pi \in \L(P)$ and
$i\in \{1,\dots,n-1\}$, set $\pi \cdot \tau_i:=\pi s_i$ if $\pi_i$ and
$\pi_{i+1}$ are incomparable in $P$, and $\pi \cdot \tau_i:=\pi$ otherwise.
That is, we interchange the two vertices at positions $i$ and $i+1$ if
$\pi s_i$ is still a linear extension of $P$ (see
Figure~\ref{figure.action} for examples). Note that, when $P$ is the
antichain (i.e. there are no relations between any elements of $P$),
$\L(P)=S_n$ and $\tau_i=s_i$.

A natural analogue of picking a card at position $i$ and inserting it
at position $j$ is then given by the operator $T_{i,j}$ defined as
\[
T_{i,j}:= \begin{cases}
\tau_i \tau_{i+1} \cdots \tau_{j-1} & \text{if $i\le j$,} \\ 
\tau_{i-1} \tau_{i-2} \cdots \tau_j & \text{if $i> j$,}
\end{cases}
\]
(see Figure~\ref{figure.action} for examples).
We define the \emph{$P$-random-to-random shuffle} 
of $\pi\in \L(P)$ by picking two positions $i,j\in [n]$
uniformly at random and applying $T_{i,j}$. The transition matrix $M=M_P$
of the $P$-random-to-random shuffle Markov chain is then given by
\begin{equation}
\label{equation.transition matrix}
M(\pi,\pi') = \frac{1}{n^2} \bigl| \{ (i,j) \in [n]\times [n] \mid \pi'= \pi \cdot T_{i,j} \} \bigr|\;.
\end{equation}

Again, if $P$ is the antichain, this reduces to the usual random-to-random transition matrix
(see for example~\cite[Section 5]{reyes.2002}), in which case the entries are $1/n$ on the diagonal,
$2/n^2$ if $\pi$ and $\pi'$ differ by a simple transposition, $1/n^2$ if they differ by any other transposition,
and zero otherwise.

\begin{figure}
  \begin{displaymath}
    \vcenter{\hbox{\includegraphics[scale=.8]{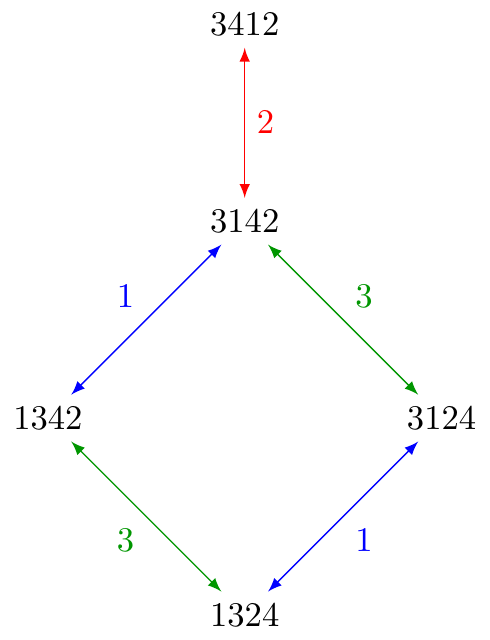}}}\qquad
    \vcenter{\hbox{\includegraphics{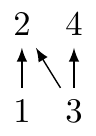}}}\qquad
    \vcenter{\hbox{\includegraphics[scale=.8]{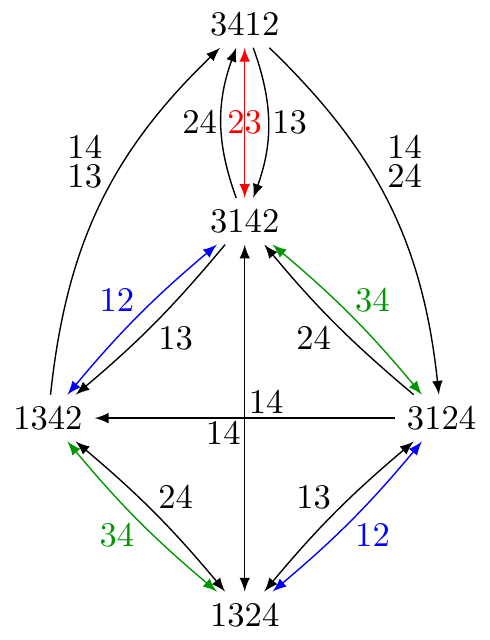}}}
  \end{displaymath}
  \caption{Graph of the action of the operators $\tau_i$ (on the left)
    and of the operators $T_{i,j}$ for $i < j$ (on the right) on the
    linear extensions of the poset depicted in the middle. 
    All loops are omitted (in particular, the
    operators $T_{i,i}$ do not appear). Using the fact that
    $T_{i,j}=T_{j,i}^{-1}$, the action of the operators $T_{i,j}$ for
    $i>j$ can be read off by reversing edges; for example,
    $3412 \cdot T_{4,1}=1342$. The operators $T_{i,i+1}=\tau_i$ are
    highlighted the same way as in the left picture.
    The graph on the right illustrates the $P$-random-to-random
    shuffle Markov chain.}
\label{figure.action}
\end{figure}

Note that the usual random-to-random transition matrix is $M=N^t N$, where $N$ is the transition matrix
for the random-to-top shuffle. Similarly for a general poset $P$, the above transition matrix $M_P=N_P^t N_P$ 
where $N_P$ is the transition matrix of the promotion Markov chain~\cite{ayyer_klee_schilling.2014}.

Our main conjecture is about the second-largest eigenvalue of the $P$-random-to-random shuffling transition matrix.
A poset $P$ is called \emph{disconnected} if its Hasse diagram is a disconnected graph, that is, there exist two elements
$x,y\in P$ such that there is no string of elements $x=x_0,x_1, x_2,\ldots, x_k=y$ such that $x_i$ and $x_{i+1}$ are comparable in $P$ for all $0\le i<k$.
We call $P$ trivial if it is the chain, since the chain has only one linear extension and there is no second-largest eigenvalue. 
For all other posets, we have the following conjecture.

\begin{conjecture}
\label{conjecture.main}
Let $P$ be a finite nontrivial poset of size $n$ and $M_P$
be the transition matrix~\eqref{equation.transition matrix}
of the $P$-random-to-random shuffling. Then, the second-largest eigenvalue $\lambda_2(P)$ of $M_P$ is bounded above by
\begin{equation}
\label{equation.second largest}
	\lambda_2(P) \le \left(1+\frac{1}{n} \right) \left(1-\frac{2}{n} \right)
\end{equation}
with equality if and only if the poset $P$ is
disconnected. Furthermore, all eigenvalues are non-negative.
\end{conjecture}

This statement includes the usual random-to-random Markov chain, when
$P$ is an antichain. For this special case, Uyemura-Reyes~\cite{reyes.2002}
proved that $(1+1/n)(1-2/n)$ appears as an eigenvalue of the
irreducible representation of $S_n$ indexed by the partition $(n-1,1)$
and therefore of $M$ itself. A complete proof of all eigenvalues of the random-to-random
Markov chain (which includes a proof that $(1+1/n)(1-2/n)$ is indeed the second largest eigenvalue) 
was recently given by Dieker and Saliola~\cite{dieker_saliola.2014}.

Usually precise formulas for the eigenvalues -- or even just the
second largest eigenvalue -- of the transition matrix are hard to obtain.
Since the largest eigenvalue of $M_P$ is one, the second largest eigenvalue determines the spectral gap
and the relaxation time. The mixing time can be bounded in terms of the relaxation time and we show
in Section~\ref{section.mixing} that the mixing time of the $P$-random-to-random Markov chain is
$O(n^2 \log n)$. This gives a faster sampling of the linear extensions of finite posets
than for the Markov chain of Bubley and Dyer~\cite{bubley_dyer.1999} (though, as discussed in Section~\ref{section.mixing},
the computational complexity is the same).

This paper is organized as follows. In Section~\ref{section.mixing} we explain the implications of
Conjecture~\ref{conjecture.main} to the relaxation and mixing time. In Section~\ref{section.evidence}
we provide computational evidence for the conjecture and prove it in special cases.

\subsection*{Acknowledgements}
We would like to thank Franco Saliola and Mike Zabrocki for helpful discussions. Many thanks to Peter Gibson
for suggesting the use of Wilkinson's theorem~\cite{wilkinson.1965, dancis_davis.1987} to prove the conjecture,
and Vincent Delecroix for helpful comments on the manuscript.

The authors would like to thank the Indian Institute for Science in Bangalore for hospitality, where part of this work was done. 
AA would like to acknowledge support in part by a UGC Centre for Advanced Study grant.
AS and NT were partially supported by the grant
DSTO/PAM/GR/1171 from the National Mathematics Institute for their trip to Bangalore.
AS was partially supported by NSF grant OCI--1147247 and the VI--MSS program sponsored by ICERM.
NT was partially supported by NSF grant OCI--1147247 for his visit to UC Davis, where part of this work was performed.

\section{Mixing time}
\label{section.mixing}

In this section we explain the implications of Conjecture~\ref{conjecture.main} for the mixing time of the 
$P$-random-to-random shuffling for a finite poset $P$.

\subsection{Spectral gap to mixing time}
\label{subsection.gap}

We begin by stating the main result of this section.

\begin{theorem}[Bound on the mixing time]
\label{theorem.mixing time}
  Let $P$ be a finite nontrivial poset of size $n$.
  Assuming Conjecture~\ref{conjecture.main}, the mixing time of the $P$-random-to-random Markov
  chain is $O(n^2 \log n)$.
\end{theorem}

Recall that the \emph{total variation distance} between two
distributions $\sigma$ and $\tau$ on a set $S$ is given by
\begin{displaymath}
  \|\sigma-\tau\| = \max_{A\subset S} |\sigma(A)-\tau(A)|\,.
\end{displaymath}

Let $M$ be the transition matrix of an irreducible and aperiodic discrete Markov
chain on a set $S$ (in particular $M$ is a row stochastic matrix). For
$x,y\in S$, the entry $M^k(x,y)$ denotes the probability of going from $x$ to
$y$ in $k$ steps of the Markov chain. The convergence of the Markov
chain to its stationary distribution $\pi$ is measured by
\begin{displaymath}
  d(k) = \max_{x\in S} || M^k(x,\cdot) - \pi(\cdot)||\,.
\end{displaymath}
By the Perron--Frobenius theorem, the convergence is exponential, that is, $d(k)\approx
c\lambda^k$ for some constants $c$ and $0\leq\lambda<1$.

The \emph{mixing time} of a Markov chain is defined by choosing some
$\varepsilon>0$, and setting $t_{\mix}(\varepsilon) := \inf \{t\geq 0\suchthat d(t) \leq \varepsilon\}$.

Assume now that $M$ is further reversible. Write
$\lambda_2=\lambda_2(M)$ for the second largest eigenvalue of $M$, and let
$\gamma:=1-\lambda_2$ be the \emph{spectral gap} of $M$. Finally, let
$t_\rel:=\frac1\gamma$ be the \emph{relaxation time}. Then, by~\cite[Theorem 12.4 and 12.3]{levin_peres_wilmer.2009}
\begin{equation}
\label{equation.inequality}
  (t_{\rel}-1)\log\left(\frac1{2\varepsilon}\right) \leq t_\mix(\varepsilon)\leq \log\left(\frac1{\varepsilon\pi_{\min}}\right)t_{\rel}\,,
\end{equation}
where $\pi_{\min}$ is the minimal value of the stationary distribution.

\begin{proof}[Proof of Theorem~\ref{theorem.mixing time}]
By Conjecture~\ref{conjecture.main} we have that $\gamma \ge \frac{1}{n} + \frac{2}{n^2}$
(with equality for disconnected posets). This implies that $t_\rel \le \frac{n^2}{n+2} \sim n$.
In addition, $\pi_{\min} = \frac{1}{|\mathcal{L}(P)|} \ge \frac{1}{n!} \sim (\frac{e}{n})^n$.
Using these in~\eqref{equation.inequality} we obtain
\begin{displaymath}
\log\left(\frac{1}{\pi_{\min}}\right) = O(n \log n)
\end{displaymath}
which implies $t_\mix(\varepsilon) =O(n^2 \log n)$ as desired.
\end{proof}

For the random-to-random Markov chain (i.e. when $P$ is the antichain) Diaconis and 
Saloff-Coste~\cite{diaconis_saloff-coste.1995} showed that the mixing time is $O(n \log n)$
as compared to $O(n^2 \log n)$ of Theorem~\ref{theorem.mixing time}.
The proof connects the mixing time of the random-to-random shuffle to a transposition Markov chain,
which was analyzed using the powerful machinery of $S_n$-representation theory,
since the transposition $t_{i,j}$ in $S_n$ can be written as a product of simple transpositions as
\begin{equation}
\label{equation.tij}
	t_{i,j}= \begin{cases} s_i s_{i+1} \cdots s_{j-2} s_{j-1} s_{j-2} \cdots s_i
	& \text{if $i\le j$,}\\
         s_{i-1} s_{i-2} \cdots s_{j+1} s_j s_{j+1} \cdots s_{i-1} & \text{otherwise.}
        \end{cases}
\end{equation}
In our setting of a general poset $P$, we can replace each $s_i$ by $\tau_i$ in~\eqref{equation.tij}.
This yields an analogue of the transposition Markov chain for general $P$. However, the operators $\tau_i$ 
do not satisfy the braid relations $(s_is_{i+1})^3=1$, which are replaced by the relation 
$(\tau_i\tau_{i+1})^6=1$~\cite{stanley.2009}.
Therefore, the representation theory of $S_n$ is no longer at our disposal. Instead, we would need
to study finite-dimensional representations of quotients of the infinite Coxeter group defined by the above relations
(which to our knowledge have not yet been studied in detail).

\subsection{Application to sampling of linear extensions}
Sampling the set of linear extensions of a poset is a fundamental problem in the theory of ordered sets and has important
implications in computer science by virtue of its connections with sorting. Bubley and Dyer~\cite{bubley_dyer.1999} define
a Markov chain on the set of linear extensions which exhibits a mixing time of $O(n^3 \log n)$, which they prove using the
method of path coupling. The mixing time of our chain of Theorem~\ref{theorem.mixing time} is at most $O(n^2 \log n)$,
which we believe can be even further improved to $O(n \log n)$ in analogy with the random-to-random chain
as explained in the previous subsection. 

\begin{conjecture}
\label{conjecture.improve}
  Let $P$ be a finite nontrivial poset of size $n$.
  The mixing time of the $P$-random-to-random Markov chain is $O(n \log n)$.
\end{conjecture}

We have verified this conjecture on posets of small sizes by the following procedure. We fixed $\varepsilon = 0.1$. 
For sizes $n=5,\dots,9$, we constructed 100 random posets in each case by adding order relations between pairs 
of elements with probability $1/n$ subject to consistency. Note that we need the probability to be small to avoid 
getting trivial posets. For each poset, we started with the identity element in the set of linear extensions, and calculated 
the mixing time. We averaged over the mixing time over 100 posets for each $n$, plotted this average as a 
function of $n$, and noted that the trend is consistent with $n \log n$.
Further evidence is given in Section~\ref{subsection.diameter}.

For a fair comparison, we now also compare the algorithmic complexity
of the chain by Bubley and Dyer and our chain. With our notation, and acting on the right rather
than on the left, each step of the Markov chain of Bubley and Dyer can
equivalently be defined as follows: choose $c\in \{0,1\}$ uniformly at
random and an index $i\in \{1,2,\ldots,n-1\}$ according to some
specified probability distribution; if $c=1$, apply $\tau_i$ on the
current state $\pi$ and otherwise leave $\pi$ unchanged. The role of
$c$ is to make the chain aperiodic, but it does not much influence the
order of magnitude of the mixing time or the algorithmic complexity.

Given that the operators $T_{i,j}$ in our $P$-random-to-random Markov
chain are products of $O(n)$ of the operators $\tau_k$, we see that the cost of each
step of our chain is $O(n)$ times that of Bubley and Dyer's
chain. Hence a mixing time of $O(n^2\log n)$ would not be an
improvement in terms of algorithmic complexity, but $O(n\log n)$
as stated in Conjecture~\ref{conjecture.improve} definitely would be.

It can further be noted that the probability of applying $\tau_k$ at
some intermediate step of the uniform $P$-random-to-random chain
matches with the probability distribution specified by Bubley and
Dyer. Hence the two chains are very similar except that, in the
$P$-random-to-random chain, the $\tau_k$'s are forced to be grouped in
certain sequences.

\section{Evidence}
\label{section.evidence}

In this section, we provide various kinds of evidence for the validity of Conjecture~\ref{conjecture.main}.
We report on systematic numerical checks in Section~\ref{subsection.comp}. We then prove the conjecture for 
two special families of posets in Sections~\ref{subsection.directsum} and~\ref{subsection.nshape}. Our main result 
about the existence of the second-largest eigenvalue for disconnected posets is given in Section~\ref{subsection.existence}. 
We discuss further properties in Section~\ref{subsection.properties} and conclude in Section~\ref{subsection.diameter}
by showing that the diameter of the $P$-random-to-random Markov chain is small, and thus not an obstruction 
to Conjecture~\ref{conjecture.improve}.

\subsection{Computational evidence}
\label{subsection.comp}
Conjecture~\ref{conjecture.main} has been checked numerically (with double precision) for all 2451 posets of size 
$n\leq 7$, and many posets of size $8$ using \texttt{Sage}~\cite{sage,sage-combinat}. The second largest eigenvalue 
for all posets of size $n\leq 5$ is available in Figure~\ref{figure.second_largest_eigenvalue}.

\subsection{Proof for direct sums of chains}
\label{subsection.directsum}

In this section, we prove the conjecture for direct sums of chains, as
a straightforward consequence of~\cite{reyes.2002}
and~\cite{dieker_saliola.2014} for the antichain.  In particular, we
use, as proven in~\cite{dieker_saliola.2014}, that for any $n$
the random-to-random shuffling transition matrix on $n$ vertices has
second largest eigenvalue $(1+1/n)(1-2/n)$ and all eigenvalues are non-negative.

\begin{proposition}
  \label{proposition.direct_sums_of_chains}
  Conjecture~\ref{conjecture.main} holds for any direct sum of chains.
\end{proposition}
\begin{proof}
  Let $P$ be a finite poset with $|P|=n$. The operators $\tau_i$ satisfy the same relations as the elementary
  transpositions $s_i$ except that, in general, the braid relation $(s_is_{i+1})^3=1$ is replaced by
  $(\tau_i\tau_{i+1})^6=1$~\cite{stanley.2009}. The usual braid relations hold
  if and only if the poset is a direct sum of chains~\cite[Proposition 2.2]{ayyer_klee_schilling.2014}.
  Hence, in this case we can still use the representation theory of the
  symmetric group to study the Markov chain.

  Let $P$ be a direct sum of chains $C_1\oplus\cdots\oplus C_\ell$.  The
  representation of $S_n$ given by the action of $\tau_i$ on $\L(P)$
  is easily identified. Namely, not swapping two vertices that are in
  the same chain is equivalent to considering these two vertices as
  equivalent. In other words, we are considering a deck of $n$ cards
  with repetitions, each chain $C_i$ contributing $|C_i|$ identical
  cards.

  Without loss of generality, we may assume that $\mu=(|C_1|,\ldots,|C_\ell|)$ is weakly
  decreasing, and thus a partition of $n$, and that the vertices of the chains are numbered 
  consecutively starting by labeling $C_1$ with the numbers $\{1,\ldots,|C_1|\}$, labeling $C_2$ with
  $\{|C_1|+1,\ldots, |C_1|+|C_2|\}$ etc.. With this notation, $\L(P)$
  can be identified with the right quotient $S_n/S_\mu$
  of the symmetric group by the parabolic subgroup $S_\mu$, whose
  linear span is the induced representation $V_\mu:= \mathrm{Ind}_{S_\mu}^{S_n} 1$.
  The usual random-to-random Markov chain (when $P$ is the antichain) corresponds to the regular
  representation $V_{(1,\ldots,1)} = \Q S_n$.

  If $P$ is made of a single chain, we are done. Otherwise, basic
  representation theory states (see for example~\cite[Chapter 2]{sagan.2001}) that the simple module indexed by the partition $(n-1,1)$
  appears in $V_\mu$. Therefore, by~\cite{reyes.2002,dieker_saliola.2014}, $(1+1/n)(1-2/n)$ is an
  eigenvalue of $M$ on $V_\mu$.

  To conclude that $(1+1/n)(1-2/n)$ is indeed the second largest
  eigenvalue $\lambda_2(P)$ and that all eigenvalues are non-negative, we
  use that $V_\mu$ is a submodule of the regular representation
  $\Q S_n$ together with the assumption that the same statement holds there.
\end{proof}

\begin{problem}
  As discussed in Section~\ref{subsection.gap}, the $P$-random-to-random Markov chain always comes
  from a finite-dimensional representation $V$ of a (finite quotient)
  of an infinite Coxeter group. Could this be used to describe the
  simple modules $S_\mu$ appearing in $V$ and the eigenvalues of $M$
  on $V_\mu$ for general $P$?
\end{problem}

\subsection{Proof of Conjecture~\ref{conjecture.main} for $N$-shaped posets}
\label{subsection.nshape}

We now prove the bound on the second largest eigenvalue for an infinite family of posets.
Let $P$ be a poset made of two chains, where the first chain is labeled $1,2,\ldots,k$ from bottom to top and
the second chain is labeled $k+1,k+2,\ldots,n$ from bottom to top. Build $P_1$ from $P$ by adding an edge
from the bottom $k+1$ of the second chain to the top $k$ of the first chain. Similarly, let $P_2$ be obtained from 
$P$ by adding an edge from the top $k$ of the first chain to the bottom $k+1$ of the second chain.
Note that $P_2$ is in fact a single chain labeled $1,2,\ldots,n$ from bottom to top, whereas $P_1$ is an $N$-shape.

Note that the linear extensions of $P$ are a disjoint union of those of $P_1$ and of $P_2$, namely $\mathcal{L}(P)
=\mathcal{L}(P_1) \sqcup \mathcal{L}(P_2)$. Let $M_P$ be the transition matrix for $P$, and build a block diagonal matrix
$M_{P_1\cup P_2}$ from the transition matrices $M_{P_1}$ and $M_{P_2}$ (a trivial block) in such a way
that the rows and columns of $M_P$ and $M_{P_1 \cup P_2}$ are labeled in the same order 
(with the identity linear extension of $P_2$ labeling the last row and column).

\begin{example}
\label{example.N}
Let $n=4$ and $k=2$. Then
  \begin{displaymath}
    P  =\vcenter{\hbox{\includegraphics{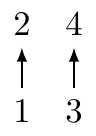}}}, \qquad
    P_1=\vcenter{\hbox{\includegraphics{Fig/P-N1.pdf}}}, \qquad
    P_2=\vcenter{\hbox{\includegraphics{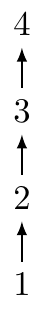}}}\,.
  \end{displaymath}
  The linear extensions of $P$ split into those of $P_1$ and of $P_2$ respectively: $\mathcal{L}(P_1)= \{3412,
  3142, 3124, 1342, 1324\}$ and $\mathcal{L}(P_2)=\{1234\}$. With this labeling
  \begin{displaymath}
    M_P = \frac1{16}\left(\!\tiny
      \begin{array}{rrrrrr}
        8 & 4 & 2 & 2 & 0 & 0 \\
        4 & 4 & 3 & 3 & 2 & 0 \\
        2 & 3 & 6 & 0 & 3 & 2 \\
        2 & 3 & 0 & 6 & 3 & 2 \\
        0 & 2 & 3 & 3 & 4 & 4 \\
        0 & 0 & 2 & 2 & 4 & 8
      \end{array}\!\right), \quad
    M_{P_1 \cup P_2} = 
    \frac1{16} \left(\!\tiny
      \begin{array}{rrrrr|r}
        8 & 4 & 2 & 2 & 0 & 0 \\
        4 & 4 & 3 & 3 & 2 & 0 \\
        2 & 3 & 6 & 1 & 4 & 0 \\
        2 & 3 & 1 & 6 & 4 & 0 \\
        0 & 2 & 4 & 4 & 6 & 0 \\
        \hline
        0 & 0 & 0 & 0 & 0 & 16
      \end{array}\!\right).
  \end{displaymath}
\end{example}

\begin{theorem}
  \label{theorem.N}
  Let $P_1$ be the $N$-shape poset as above. Then
  \[
  	\lambda_2(P_1)\leq \left(1+\frac{1}{n} \right) \left(1-\frac{2}{n} \right).
  \]
\end{theorem}

Note that the eigenvalues of $M_{P_1\cup P_2}$ are the same as the eigenvalues of $M_{P_1}$ with an
additional one. Our strategy for the proof of Theorem~\ref{theorem.N} is to relate the eigenvalues of $M_{P_1\cup P_2}$
to those of $M_P$ using Wilkinson's theorem. Denote the eigenvalues of a symmetric matrix $M\in \mathbb{R}^{n\times n}$
by
\[
	\lambda_1(M) \ge \lambda_2(M) \ge \cdots \ge \lambda_n(M)\,.
\]

\begin{theorem} [Wilkinson~\cite{wilkinson.1965, dancis_davis.1987}]
\label{theorem.wilkinson}
Suppose $B = A +\tau c c^T$, where $A\in \mathbb{R}^{n\times n}$ is symmetric, $c\in \mathbb{R}^{n \times 1}$ has
unit norm, and $\tau \in \mathbb{R}$. If $\tau\ge 0$, then
\[
	\lambda_i(B) \in [\lambda_i(A),\,\lambda_{i-1}(A)] \qquad (2\le i\le n)\,,
\]
whereas if $\tau\le 0$
\[
	\lambda_i(B) \in [\lambda_{i+1}(A),\,\lambda_i(A)] \qquad (1\le i\le n-1)\,.
\]
\end{theorem}

\begin{proof}[Proof of Theorem~\ref{theorem.N}]
Let $N_P$ and $N_{P_1 \cup P_2}$ be the random-to-top transition matrices for the posets $P$ and
$P_1 \cup P_2$, respectively. They only differ in the way linear extensions of the form
\[
	\pi^{(1,i)} = 1,2,\ldots, i-1,\; k+1,\; i, \ldots, k,\; k+2, \ldots ,n \quad  (\text{for $1\le i\le k$})\,
\]
together with the identity transition to either
\[
	\pi^{(0)} = 1,\ldots, k-1,\; k+1,\ldots, n,\; k
\]
or the identity. More concretely, the matrix $E = N_{P_1\cup P_2} - N_P$ has entries $1/n$ in row $\pi^{(0)}$
and columns $\pi^{(1,i)}$, entries $-1/n$ in the identity (last) row and columns $\pi^{(1,i)}$, entry $-k/n$ in
row $\pi^{(0)}$ and the identity (last) column, and entry $k/n$ in the identity row and column.

Now consider
\begin{multline*}
	M_{P_1\cup P_2} = N_{P_1 \cup P_2}^t N_{P_1\cup P_2} = (N_P+E)^t(N_P+E)\\
	 = M_P + E^tN_P + N_P^tE + E^tE\,.
\end{multline*}
Since $E$ and $N_P$ are known explicitly, it is straightforward to compute $E^tN_P + N_P^tE + E^tE$.
Namely, define the linear extensions
\[
	\pi^{(2,i)}= 1,\ldots, k-1,\; k+1, \ldots, i,\; k,\; i+1, \ldots, n, \quad \text{(for $k+1\le i\le n$)}.
\]
Then $E^tN_P + N_P^tE + E^tE$ contains entries $1/n^2$ (resp. $-(n-k)/n^2$) in rows $\pi^{(1,i)}$ and columns
$\pi^{(2,j)}$ (resp. identity), entries $1/n^2$ (resp. $-k/n^2$) in rows $\pi^{(2,i)}$ and columns $\pi^{(1,j)}$ (resp. identity),
and entries $-(n-k)/n^2$ (resp. $-k/n^2$ or $2k(n-k)/n^2$) in the identity row and columns $\pi^{(1,i)}$
(resp. $\pi^{(2,i)}$ or identity). When $\pi^{(1,i)}$ and $\pi^{(2,j)}$ are equal (namely for $i=k$ and $j=k+1$),
the just stated entries are added together.

The matrix $E^tN_P + N_P^tE + E^tE$ can further be written as the sum
of two rank one matrices. Namely
\[
	M_{P_1\cup P_2} = M_P + \tau cc^t - \tau dd^t\,,
\]
where $\tau=1/2k(n-k)n^2$, $c \in \mathbb{R}^{n\times 1}$ has entry $n-k$ (resp. $k$ or $-2k(n-k)$)
in rows $\pi^{(1,i)}$ (resp. $\pi^{(2,j)}$ or identity), and $d \in \mathbb{R}^{n\times 1}$ has entry $n-k$
(resp. $-k$) in rows $\pi^{(1,i)}$ (resp. $\pi^{(2,j)}$). Again, entries are added when $\pi^{(1,i)}$ and $\pi^{(2,j)}$
are equal.

We are now going to apply Wilkinson's Theorem~\ref{theorem.wilkinson} twice with
$A=M_P$, $B_1= M_P - \tau dd^t$, and $B_2 = M_{P_1\cup P_2} =  B_1 + \tau cc^t$.
Since $\tau>0$, we obtain
\begin{equation*}
\begin{split}
	&\lambda_i(B_1) \in [\lambda_{i+1}(A),\lambda_i(A)]\,,\\
	&\lambda_i(B_2) \in [\lambda_i(B_1),\lambda_{i-1}(B_1)]\,,
\end{split}
\end{equation*}
so that in particular $\lambda_3(B_2)\le \lambda_2(B_1)\le \lambda_2(A)$.
Recall that by Proposition~\ref{proposition.direct_sums_of_chains},
$\lambda_2(A) = \lambda_2(P)\le (1+1/n)(1-2/n)$. The matrix $B_2=M_{P_1\cup P_2}$
is block diagonal; the block corresponding to $P_2$ consists of a 1 and contributes an eigenvalue 1.
This implies that $\lambda_2(P_1)=\lambda_3(P_1\cup P_2) = \lambda_3(B_2)$. Combining these inequalities
yields $\lambda_2(P_1) \le (1+1/n)(1-2/n)$ as desired.
\end{proof}

In principle a similar approach can be applied for other posets besides the $N$-shape posets, by removing one cover
relation $a\prec b$ from a poset $P_1$ to obtain poset $P$ and by defining $P_2$ by adding the relation $b\prec a$
instead. The linear extensions of $P$ are again a disjoint union of the linear extensions of $P_1$ and $P_2$, so that
one can compare the block diagonal matrix for $P_1$ and $P_2$ with the matrix for $P$.  Computer experimentation
indicates that the two matrices still differ by sums of rank 1 matrices. 
However, the number thereof increases and, in general, it is harder to control bounds for $\lambda_2$ by
repeated use of Wilkinson's theorem.

\begin{problem}
Can one generalize the idea of the proof for $N$-shaped posets, possibly using Wilkinson's theorem, to prove 
Conjecture~\ref{conjecture.main}?
\end{problem}

\subsection{Existence of $(1+1/n)(1-2/n)$ as eigenvalue}
\label{subsection.existence}
Next we show that Proposition~\ref{proposition.direct_sums_of_chains} implies that 
$(1+1/n)(1-2/n)$ is an eigenvalue for any disconnected poset. To this end, let
$P$ be a finite poset and $Q$ a direct sum of chains obtained by
linearly ordering each connected component of $P$ (this is not unique, but we just pick one such
linear ordering). Define the map
\[
	\sort \colon \L(P) \to \L(Q)
\]
which takes a linear extension $\pi \in \L(P)$ and returns the
linear extension of $Q$ obtained by sorting the elements within each
connected component according to the corresponding chain.

We will now show that the $\sort$ map is a \textit{contraction} or \textit{lumping} \cite{levin_peres_wilmer.2009} 
from the Markov chain on $\L(P)$ to the Markov chain on $\L(Q)$. In our context, 
this means the following. Suppose $\Pi,\Sigma \in \L(Q)$. 
For every $\sigma, \sigma' \in \L(P)$ such that $\sort(\sigma) = \sort(\sigma') = \Sigma$, 
there exist $\pi,\pi' \in \L(P)$ (possibly equal) such that $\sort(\pi) = \sort(\pi') = \Pi$ and
$M(\sigma,\pi) = M(\sigma',\pi')$. We then define the {\em lumped} Markov chain on
$\L(Q)$ by setting $M(\Sigma,\Pi) = M(\sigma,\pi)$.
Equivalently, the ``sort'' action always commutes with the ``Markov chain'' action.

\begin{proposition}
  The map $\sort$ commutes with the operators $\tau_i$, that is, $\sort \circ \tau_i = \tau_i \circ \sort$.
  This implies that  $M_Q$ is a lumping (quotient) of $M_P$. In particular, any
  eigenvalue for $M_Q$ lifts to an eigenvalue of $M_P$.
\end{proposition}
\begin{proof}
  For $\pi \in \L(P)$, consider $\pi \cdot \tau_i$. If $\pi_i$ and $\pi_{i+1}$ belong to the same connected component
  of $P$, then $\tau_i$ acts trivially on $\sort(\pi)$. But in this case $\sort(\pi \cdot \tau_i) = \sort(\pi)$, so that indeed
  $\sort(\pi \cdot \tau_i) = \sort(\pi) \cdot \tau_i$.
  Otherwise $\pi_i$ and $\pi_{i+1}$ belong to different connected components and $\tau_i$ exchanges $\pi_i$ and $\pi_{i+1}$.
  Linearly ordering the vertices of $P$ commutes with interchanging the labels of two nodes in different components. 
  This proves the claim.
\end{proof}

\begin{corollary}
  Let $P$ be a disconnected poset with $|P|=n$. Then $(1+1/n)(1-2/n)$ is an
  eigenvalue of $M_P$.
\end{corollary}
\begin{proof}
It is well-known that
the set of eigenvalues of the contracted or lumped Markov chain
is contained in the set of eigenvalues of the original chain. 
See, for example,~\cite[Section 1.1]{caputo_liggett_richthammer.2010}.
The claim then follows by Proposition~\ref{proposition.direct_sums_of_chains}.
\end{proof}

\subsection{Further properties}
\label{subsection.properties}
We now discuss properties that are suggested by the
data in Figure~\ref{figure.second_largest_eigenvalue}. 
Recall that the {\em dual} $\overline P$ of a poset $P$ is 
defined as the poset on the same elements 
such that $x\preceq y$ in $\overline P$ if and only if $y\preceq x$ in $P$.
One first observes that the second largest 
eigenvalue for a poset and its dual coincide. 
In fact, the Markov chains are isomorphic.

\begin{proposition}
  Let $P$ be a poset and $\overline P$ its dual. Then reversing
  linear extensions yields an isomorphism between the
  $P$- and $\overline P$-random-to-random shuffle Markov chain.
\end{proposition}

\begin{proof}
  Remark first that reversal of linear extensions conjugates $\tau_i$
  on $\L(P)$ to $\tau_{n-i}$ on $\L(\overline P)$: for $\pi$ a linear
  extension of $P$ and $\overline \pi$ the linear extension of
  $\overline P$ obtained by reversing $\pi$, one has
  $\overline{\pi \cdot \tau_i} = \overline \pi \cdot \tau_{n-i}$; indeed $\tau_i$
  swaps the values $a$ and $b$ at positions $i$ and $i+1$ in $\pi$ if
  and only if $a$ and $b$ are incomparable in $P$ if and only if $a$
  and $b$ are incomparable in $\overline P$ if and only if
  $\tau_{n-i}$ swaps the values $a$ and $b$ at positions $n-i+1$ and
  $n-i$ in $\overline \pi$.
  Therefore, by composition, the reversal of linear extensions
  conjugates the transition matrix $M_{i,j}$ on $\L(P)$ to $M_{n-i+1,n-j+1}$ on
  $\L(\overline P)$. The statement follows.
\end{proof}

The data in Figure~\ref{figure.second_largest_eigenvalue} also
suggests that the second largest eigenvalue decreases when adding
comparability edges. Namely, in Figure~\ref{figure.second_largest_eigenvalue} the posets are partially ordered
with respect to inclusion of comparisons: $P\subseteq Q$ if for some
labelling of the vertices of $P$ and $Q$, $a\prec b$ in $P$ implies $a\prec b$
in $Q$ for all $a,b$. Almost everywhere $\lambda_2$ is order-reversing: 
$P\subseteq Q$ implies that $\lambda_2(P)\geq \lambda_2(Q)$. 
The smallest counterexample occurs for $n=5$ and is unique
up to duality. Namely take $P$ the second leftmost poset in the fourth
row, and $Q$ the leftmost poset in the third row.

\begin{problem}
Although the order-reversal property does not hold in general, one could hope that it holds under certain 
additional conditions. Could one use this refined property, together with the fact that disconnected posets 
are at the bottom of the inclusion poset, to prove Conjecture~\ref{conjecture.main}? For $N$-shaped posets, 
this is exactly the approach we took in the proof of Theorem~\ref{theorem.N}.
\end{problem}

\begin{figure}
  \begin{subfigure}{\textwidth}
    \caption{$n=1,2,3$}
    \centering
    \includegraphics[scale=.8]{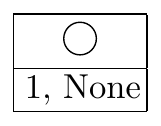}\hfil
    \includegraphics[scale=.8]{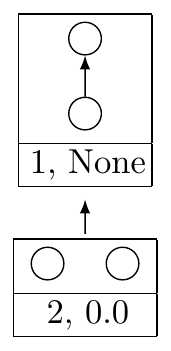}\hfil
    \includegraphics[scale=.8]{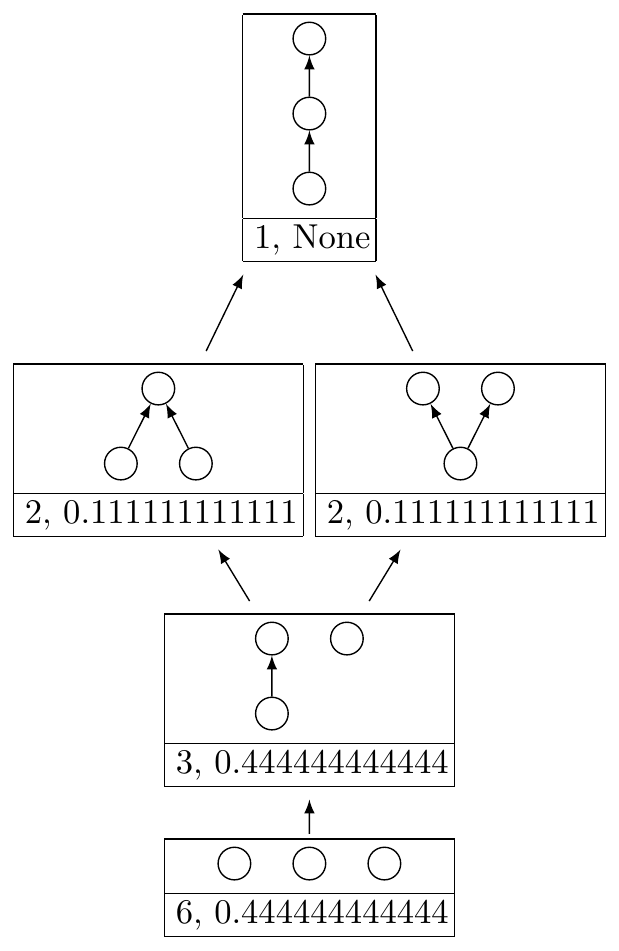}\hfil
  \end{subfigure}
  \caption{The second largest eigenvalue for all posets of size $1\le
    n \le 5$, preceded each time by the number of linear
    extensions. For each $n$, the posets are ordered by inclusion of
    comparison relations.  For comparison, for $n=2,3,4,5$, the
    conjectured bound $(1+\frac1n)(1-\frac2n)$ on the second largest
    eigenvalue is given by $0.0,\,0.44\overline{4},\,0.625,\,0.72$,
    respectively.}
  \label{figure.second_largest_eigenvalue}
\end{figure}

\begin{figure}
  \ContinuedFloat
  \begin{subfigure}{\textwidth}
    \caption{$n=4$}
    \includegraphics[height=\textheight]{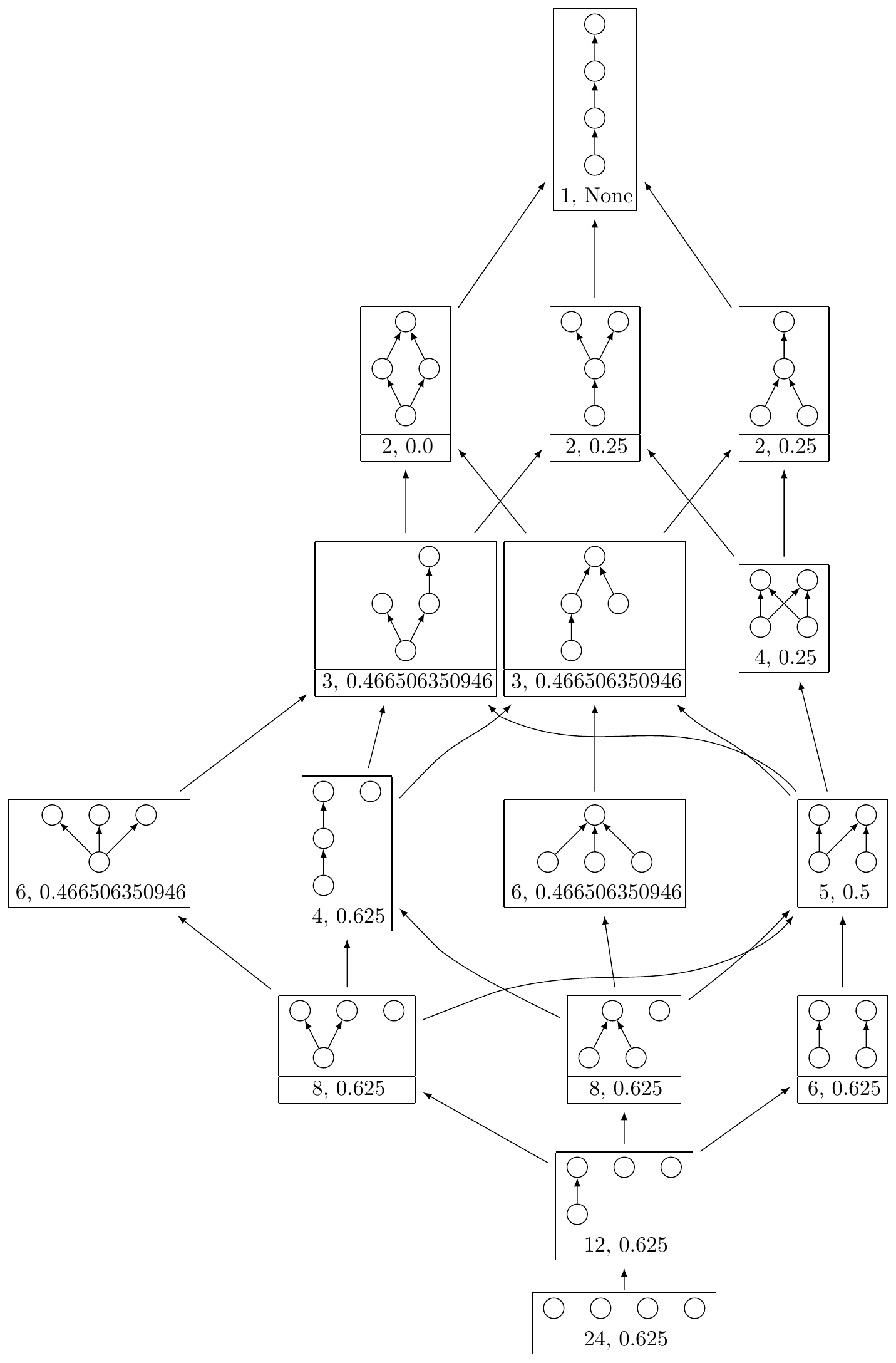}\hfil
  \end{subfigure}
\end{figure}

\begin{figure}
  \ContinuedFloat
  \begin{subfigure}{\textwidth}
    \caption{$n=5$}
    \begin{bigcenter}
      \includegraphics[width=.9\paperwidth]{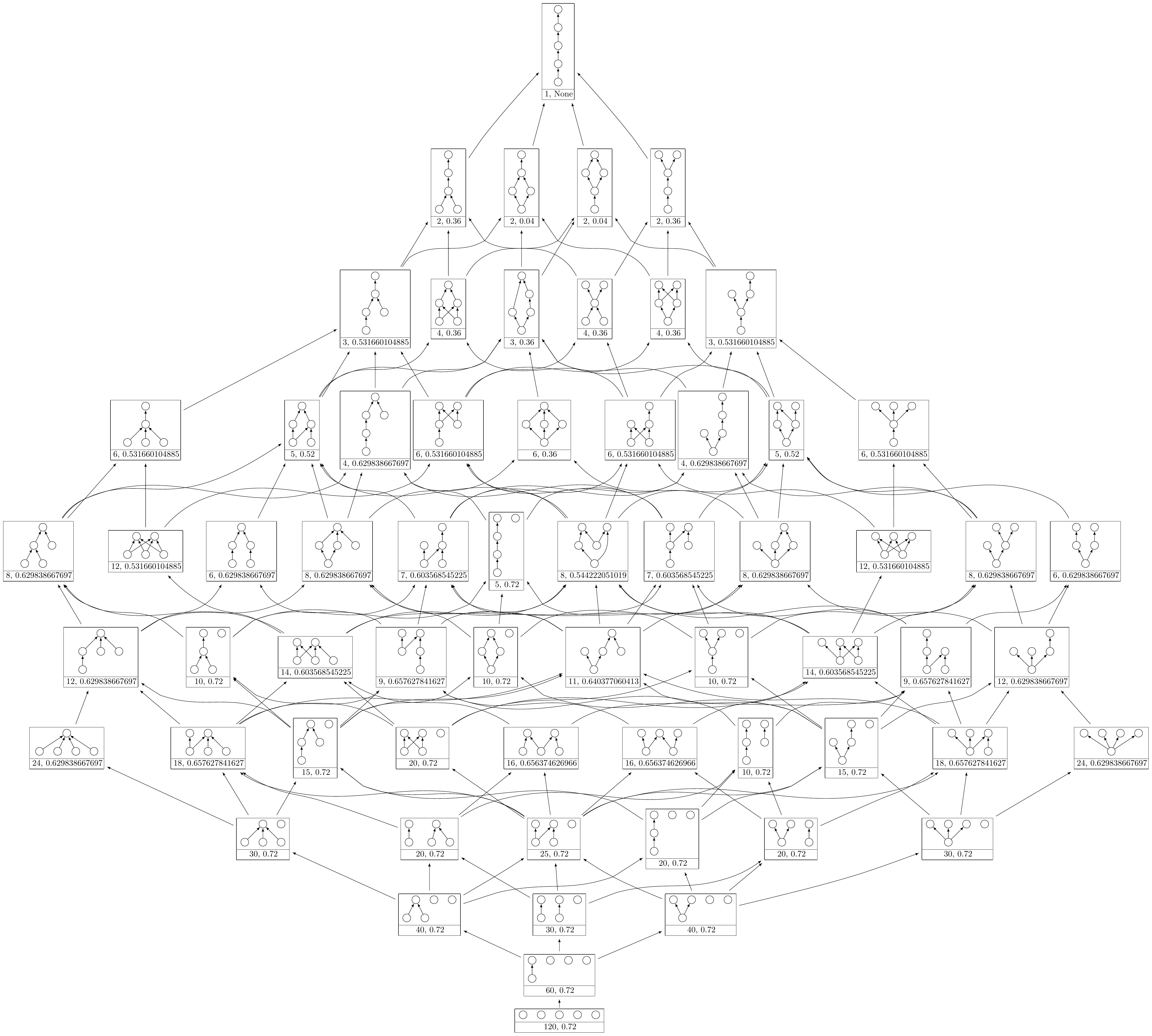}\hfil
    \end{bigcenter}
  \end{subfigure}
\end{figure}

\subsection{On the diameter of the Markov chain}
\label{subsection.diameter}

A necessary condition for a small mixing time is that the states of the system are not
too far away from each other; namely $t_{\mix}(\varepsilon) \geq L/2$,
where $L$ is the \emph{diameter} of the graph of the Markov chain,
that is the largest graph distance between any two states (see
e.g.~\cite{levin_peres_wilmer.2009}, Equation (7.3)).
We now show that the diameter of the $P$-random-to-random Markov chain
is small enough, despite its size, to not be an obstruction to
Conjecture~\ref{conjecture.improve}.

\begin{proposition}
  Let $P$ be a poset of size $n$. Then the diameter of the
  $P$-random-to-random shuffle Markov chain is bounded above by $n$.
\end{proposition}

\begin{proof}
  Take two linear extensions $\pi$ and $\pi'$. Without loss of
  generality, we may assume that $P$ is labeled by $1,\ldots,n$
  according to the second linear extension $\pi'$. Let us try to sort
  $\pi$ to $\pi'$ using the operators $T_{i,j}$.

  Assume that the first $i-1$ positions are sorted, that is,
  $\pi$ starts with $1,2,\ldots,i-1$. We want to move $i$, which is at
  some position $j\geq i$ in $\pi$ to the $i$-th position. Given that
  $\pi'$ is a linear extension and starts with $1,\ldots,i$, the
  values at positions $i,\ldots,j-1$ are incomparable to $i$. Therefore
  applying the operator $T_{j,i}$ moves $i$ at position $j$ to position $i$, as
  desired.

  Applying induction on $i$ shows that $\pi'$ is reachable from $\pi$
  by the application of at most $n$ operators.
\end{proof}

\bibliographystyle{alpha}
\bibliography{main}
\end{document}